\documentclass{article}
\usepackage{diagbox}
\usepackage{nicematrix}
\usepackage{amsmath}
\usepackage{graphicx}
\usepackage{amssymb}
\usepackage{amsmath}
\usepackage{amsfonts}
\usepackage{amsthm}
\usepackage{float}
\usepackage{makecell}
\usepackage{xcolor}
\usepackage[english]{babel}
\usepackage{caption}
\usepackage[affil-it]{authblk}

\usepackage{soul}

\newtheorem{theorem}{Theorem}
\newtheorem{lemma}{Lemma}
\newtheorem{proposition}{Proposition}

\newtheorem{remark}{Remark}

\newcommand{\alert}[2][magenta]%
{{\color{#1}\mbox{[\hspace{-0.4ex}[}#2\mbox{]\hspace{-0.4ex}]}}}

\makeatletter
\newcommand{\ostar}{\mathbin{\mathpalette\make@circled\star}}
\newcommand{\make@circled}[2]{%
  \ooalign{$\m@th#1\smallbigcirc{#1}$\cr\hidewidth$\m@th#1#2$\hidewidth\cr}%
}

\newcommand{\smallbigcirc}[1]{%
  \vcenter{\hbox{\scalebox{0.77778}{$\m@th#1\bigcirc$}}}%
}
\makeatother
\title{Completely regular codes with covering radius $1$ and the second eigenvalue in $3$-dimensional Hamming graphs}

\author{Ivan Mogilnykh, Anna Taranenko and Konstantin Vorob'ev
\footnote{The work was funded by the Russian Science Foundation (22-11-00266),
http://rscf.ru/project/22-11-00266/.

**The authors are with the Sobolev Institute of Mathematics, Novosibirsk, Russia.
}$^{\,\,\,,**}$}

\begin{document}
\maketitle
 
\begin{abstract} 
We obtain 
a classification of the completely regular codes with covering radius $1$ and the second eigenvalue in the Hamming graphs $H(3,q)$ up to $q$ and intersection array. Due to   works of Meyerowitz, Mogilnykh and Valyuzenich,  our result  completes  the classifications of completely regular codes   with covering radius $1$ and the second eigenvalue in the Hamming graphs $H(n,q)$ for any $n$ and completely regular codes with covering radius $1$  in $H(3,q)$.

\end{abstract}


\section{Introduction}
The purpose of this study is completing the classification of the completely regular codes in $H(n,q)$ with covering radius $1$ and the second eigenvalue that was initiated in \cite{MV}. Completely regular codes with covering radius $1$ are  known under various different names, e.g. equitable $2$-partitions, perfect $2$-colorings, and intriguing sets  \cite{BLP}, \cite{BKMTV}, \cite{MV}. The concept of a completely regular code, suggested by Delsarte \cite{Del73},  is traditionally considered in classical association schemes. For a survey on completely regular codes in  Hamming and Johnson graphs we refer to \cite{ZinRifBorg}. In \cite{BKMTV} a more specific case of  covering radius $1$ (equitable $2$-partitions) in Hamming graphs is studied in details.      

In recent years, several new approaches for studying special classes of completely regular codes in different association schemes were suggested. 
Some papers focus on codes of small strength \cite{Vor20},  \cite{FI}, \cite{MV},  other target  codes with minimum possible eigenvalue \cite{MV2} or linear codes of small covering radius \cite{BRZ2}.

We grade $q$-ary
completely regular codes by their strengths as 
 orthogonal arrays. It is well-known that the strength is determined by the largest eigenvalue of the code in the following manner. A completely regular code has strength $i$ if its maximum eigenvalue, different from the graph valency, has index $i+1$ as eigenvalue of  the Hamming graph \cite{Del73}. The completely regular codes of strength $0$ in the Hamming graphs were  classified up to isomorphism by Meyerowitz in \cite{Mey}.

We say that the position $j$ is {\it essential} for a code (set) $C$ in $H(n,q)$ if there are vertices from $C$ and the complement of $C$ different only in  $j$th position.  The eigenvalues of the graph $H(n,q)$  are $$\lambda_i(n,q) = n(q-1) - qi$$ with multiplicity ${n \choose i} (q-1)^i$, $i = 0, \ldots, n$ (see, for example, \cite{BCN}). We say that a code in $H(n,q)$ is {\it reduced}  if all its positions are essential.
The following fact is well-known, see e.g. \cite[Theorem 4.5]{BKMTV},   \cite[Proposition 33(ii)]{ZinRifBorg},  \cite[Lemma 2]{MV}. 
\begin{proposition}\label{P:3}  Let $C$ be a code in $H(n,q)$ with nonessential position  
$j$. The code $C$ is completely regular with covering radius $\rho$, intersection array $A$ and eigenvalues $\{\lambda_i(n,q):i\in I\}$ if and only if
the code in  $H(n-1,q)$ obtained by deleting $j$th position in the
tuples of $C$ is completely regular with  covering radius $\rho$, intersection array $A$ and eigenvalues $\{\lambda_i(n-1,q):i\in I\}$.

\end{proposition}

 Two constuctions of completely regular codes with covering radius $1$ were introduced in \cite{MV}.
One is based on alphabet liftings of two cycles in the Hamming graph $H(4,2)$, see also \cite{BKMTV}.
The idea behind the second method, called permutation switching construction is as follows. We choose a specially chosen completely regular code in $H(2,q)$ and 
add nonessential positions. Then  we apply coordinate permutations (switchings) to certain subcodes of this  code. The obtained code has from $3$ 
to $\frac{q}{2}+1$ essential positions so it is nonisomorphic to the original code with two  essential coordinates, while both codes have the same parameters.
The following theorem was proven for the completely regular codes  with the second eigenvalue, the details on both mentioned constructions below could be found  in \cite{MV}. 
\begin{theorem}\cite{MV}\label{T:reduce}
All reduced completely regular codes in $H(n,q)$ with covering radius $1$, second eigenvalue and at least four  essential positions are obtained by permutation switchings  or  alphabet liftings of a cycle in $H(4,2)$. The intersection arrays of these codes have only even entries.
\end{theorem}
The theorem above reduces the classification to the case of $H(3,q)$, which we tackle in this work. Our study shows that there are many nonisomorphic codes with the same parameters ($q$ and intersection array), so we narrow the goal to  classification up to parameters. Meanwhile, we still obtain complete description of these objects in certain cases.
 
 The paper is organized as follows. We provide necessary definitions and background results in Sections \ref{S:Def} and \ref{S:3}. Four constructions (Constructions A-D) of completely regular codes with new parameters are given in Section \ref{S:Constr}, two of which (Construction A and B) were known in the literature, while  Constructions C and D are new and  provide a rich family of new codes, sufficient to cover all open parameter cases. Construction D can be considered as a  generalization of the splitting construction II from~\cite{BKMTV} for $3$-dimensional Hamming graphs.

 As a simple consequence of Construction $A$ we characterize all completely regular codes for odd $q$ up to parameters (Proposition \ref{p:odd q}). 
 Construction C covers all codes of sizes at least $\frac{q^3}{4}$ in the Hamming graph $H(3,q)$ up to parameters for even $q$.
 In Section \ref{S:CRCsclique}, we show that  Construction D gives a complete description for completely regular codes with covering radius $1$ that are unions of disjoint cliques in $H(3,q)$. 
 
Furthermore, we make use of an eigenvalue technique  \cite{MV},  \cite{Vor20}, \cite{V1} in Section \ref{S:Deriv} and show that Construction D describes all completely regular codes with the second eigenvalue, even $q$, and odd entries in intersection array when the size of a code is less than a quarter of the order of the graph. This fact, combined with other constructions, is an essential ingredient of the  characterization of  completely regular codes with covering radius $1$ in $H(3,q)$ up to parameters given in Section 7.
The results of this work and Theorem \ref{T:reduce} imply  the description of the parameters of completely regular codes with covering radius $1$ and the second eigenvalue in $H(n,q)$ (Theorem \ref{T:Hnq}).

\section{Definitions}\label{S:Def}
Let $C$ be a {\it code} in a regular graph $\Gamma$, i.e. a set of its vertices. The complement of the set $C$ in $\Gamma$ is denoted by $\overline{C}$.
A vertex $x$ is in $C_i$ if the minimum of the distances between $x$ and the vertices of $C$ is $i$.
The maximum of these distances is called {\it the covering radius} of $C$ and is denoted by $\rho$. The  {\it distance partition} of the vertices of $\Gamma$ with respect to $C_0=C$ is  $\{C_i:i\in \{0,\ldots,\rho\}\}$.

A code $C$ is called {\it completely regular} (shortly, CRC) if there are numbers
$\alpha_0,\ldots,\alpha_{\rho}$,  $\beta_0,\ldots,\beta_{\rho-1}$, $\gamma_1,\ldots,\gamma_{\rho}$
such that for any $i$,  any vertex of $C_i$ is adjacent to exactly $\alpha_i$, $\beta_i$, and $\gamma_i$ vertices of $C_{i-1}$, $C_{i}$, and $C_{i+1}$
respectively. Note that  $\alpha_0,\ldots,\alpha_{\rho}$ can be found from the remaining numbers and the valency of the graph. Similarly to distance-regular graphs \cite{BCN}, the array 
  $\{\beta_0,\ldots,\beta_{\rho-1}; \gamma_1,\ldots,\gamma_{\rho}\}$ is called the {\it intersection array} of the completely regular code $C$.

For a completely regular code $C$ consider the three-diagonal $(\rho+1)\times(\rho+1)$ {\it intersection matrix} $M$
such that $M_{i,j}$ equals the number of vertices of $C_j$ adjacent to a fixed vertex of $C_i$.
Its  eigenvalues  are called the {\it eigenvalues of the completely regular code} $C$.

It is well-known (see e.g.
\cite[Theorem 4.5]{CDZ}) that the eigenvalues of the intersection matrix $M$ are necessarily eigenvalues of the adjacency
matrix of the  graph.

\begin{remark}
As all completely regular codes we consider below have covering radius $1$, throughout the paper we simplify the
notations for the intersection numbers. We denote $\gamma_1$ by $\gamma$ and $\beta_0$ by $\beta$.
\end{remark}

The following statement is well-known, see e.g. \cite{MV}.
\begin{proposition}\label{P:1} Let $C$ be a completely regular code with $\rho=1$ in a
$k$-regular graph $\Gamma$. Then $\alpha_0+\beta=\gamma+\alpha_1=k$, $|C|=|V(\Gamma)|\gamma/(\gamma+\beta)$, and the eigenvalues of the
code $C$ are  $k$ and  $k - (\gamma+\beta)$.   
\end{proposition}
 
From the statement above, we see that the intersection matrix $M=\begin{pmatrix}\alpha_0 & \beta\\\gamma & \alpha_1\end{pmatrix}$ of a completely regular code with $\rho=1$ in a $k$-regular graph $\Gamma$ can be reconstructed given the element $\gamma$ of the  intersection array and its eigenvalue $\lambda \neq k$.   It is easy to see that a code is CRC with 
 $\rho=1$ if and only if its complement is CRC with $\rho=1$. In view of this fact, we  assume that $\gamma \leq  \beta$ and $|C|\leq \frac{|V(\Gamma)|}{2}$ throughout the paper (otherwise switch over $C$ and $\overline{C}$).

The vertex set of the {\it Hamming graph}
$H(n,q)$ is the Cartesian $n$th power of a set ${\cal A}$ of size
$q$, vertices $x$ and $y$ are adjacent if they  differ in
exactly one coordinate position.

For a given $a\in {\cal A}$ the set $\{x\in {\cal A}^n:x_i=a\}$ is called a  {\it hyperface} of direction $i$.
Any maximum clique of a Hamming graph $H(n,q)$ is a set of $q$ tuples having
distinct symbols in a fixed position $i$.  We refer to $i$ as the {\it codirection} of the clique to emphasize the dualistic nature of cliques and hyperfaces. We note the following  obvious property in the Hamming graphs.
\begin{equation}
\label{eq:obvious}
\mbox{A set of } k \mbox{ cliques of codirection } i \mbox{ meets any hyperface of direction } i \mbox{ in } k \mbox{ vertices.}
\end{equation}

The {\em Cartesian product} $\Gamma\square \Gamma'$ of graphs $\Gamma$ and $\Gamma'$ is a graph with the vertex set $V(\Gamma)\times V(\Gamma')$, and
any two vertices $(u,u')$ and $(v,v')$ are adjacent if and only if either
$u=v$ and $u'$ is adjacent to $v'$ in $\Gamma'$, or
$u'=v'$ and $u$ is adjacent to $v$ in $\Gamma$. We see that the Hamming graph $H(n,q)$ is the Cartesian $n$th power of the complete graph $H(1,q)$.

\section{Stochastic properties of CRCs with $\rho=1$ in the  Cartesian products of complete graphs}\label{S:3}

In this section we recall one of the most known properties of the completely regular codes in Hamming graphs which relate it to an orthogonal array of certain  "strength" \cite{Del73}. A similar  fact could be formulated for a more general class of graphs  such as the Cartesian product of  complete graphs with different orders, which are not necessarily distance-regular. In case of these graphs, we   rather use a general term "stochastic set" that finds similarity with a  consonant matrix class.

The maximal cliques of the graph $H(1,q)\square H(1,q')$ are ${\cal A}\times i$ and $j\times {\cal A}'$  for $i\in {\cal A}'$, $j\in {\cal A}$, where ${\cal A}$ and ${\cal A}'$ are the alphabet sets of  $H(1,q)$ and $H(1,q')$, respectively. We see that the CRCs   with the minimum eigenvalue in the graph $H(1,q)\square H(1,q')$ are evenly distributed by the maximal cliques of the graph.

\begin{proposition}\label{P:2}
1. \cite[Proposition 2]{MV} A code $C$ in the graph
$H(1,q)\square H(1,q')$ is completely regular with $\rho=1$ and eigenvalue $-2$ 
if and only if for any maximal clique $K$ the number $|K\cap C|$ depends only on  $|K|$. Moreover, for the parameter $\gamma$ of such a code $C$ we have that 
$\frac{|K\cap C|}{|K|}=\frac{\gamma}{q+q'}$ for any maximal clique $K$ of $H(1,q)\square H(1,q')$.

2. A completely regular code in the graph
$H(1,q)\square H(1,q')$ with  $\rho=1$, eigenvalue $-2$, and $\gamma$  exists if and only if  $\frac{\gamma}{q+q'}q$ is an integer, $0<\frac{\gamma}{q+q'}q< q$.
\end{proposition}
\begin{proof} 2.
Necessity follows from the first statement.

Sufficiency. 
Let $g$
be a  permutation of order $q$ on the vertex set of $H(1,q)$.
Consider a  set $D$ of vertices of $H(1,q)$ of size $\frac{q\gamma}{q+q'}$ and its $q'$ images  $g^{i\frac{q\gamma}{q+q'}}(D)$ for $i=0,\ldots,q'-1$.
The desired CRC with $\rho=1$ in $H(1,q)\square H(1,q')$ is obtained as 
$$\{ (g^{i\frac{q\gamma}{q+q'}}(j),i):j\in D,i=0,\ldots,q'-1\}.$$
Indeed, it is easy to see that any maximal clique of size $q$ or $q'$ of the graph $H(1,q)\square H(1,q')$ contains either $\frac{q\gamma}{q+q'}$ or 
 $\frac{q'\gamma}{q+q'}$ vertices of $C$ and the result follows from the first statement.

\end{proof}

A code   $C$ in the graph $H(1,q)\square H(1,q')$ is called $(a,b)$-{\it stochastic} if any maximal clique in $H(1,q)\square H(1,q')$ of size $q$ contains $a$ vertices of $C$ and 
any maximal clique of size $q'$ contains $b$ vertices of $C$. We include the case $a=q$ here, i.e. when $C$ coincides with the vertex set of the graph. If $a<q$, in view of the above statement we see that $C$ is $(a,b)$-stochastic if and only if  $C$ is CRC with $\rho=1$ and eigenvalue $-2$; its parameter $\gamma$ is $a+b$, where $a=\frac{\gamma q}{q+q'}$ and $b=\frac{\gamma q'}{q+q'}$. 
In turn, an $(a,b)$-stochastic subset of  
$H(1,q)\square H(1,q')$ exists if and only if 

\begin{equation}\label{c:stochastic}aq'=bq.\end{equation}

In a similar fashion, we call a subset $C$ of $H(1,q)\square H(1,q')$ $(a,*)$-{\it stochastic}  or $(*,b)$-{\it stochastic}
if any maximal clique of size $q$ and  $q'$ respectively contains exactly $a$ and $b$ vertices from $C$ respectively. Obviously $(a,*)$- and $(*,b)$-stochastic set is  $(a,b)$-stochastic.

It is well-known that if the characteristic
function of a code in $H(n,q)$ is orthogonal to any
$\lambda_j(n,q)$-eigenfunction for all $j$ such that $1\leq j\leq
t$, then the code is a $t$-orthogonal array (see e.g.
\cite[Theorem 4.4]{Del73}). In particular, in the case when
$C$ is a CRC in $H(n,q)$ with
the second eigenvalue, the code $C$ is evenly distributed by the Hamming subgraphs $H(n-1,q)$. We have the following for $n=3$.

\begin{proposition}\label{p:OA}

1. \cite[Theorem 4.4]{Del73} Let $C$ be a completely regular code  with $\rho=1$, 
eigenvalue $\lambda_2(n,q)$, and $\gamma$ in $H(3,q)$. 
Then for
any $i\in \{1,2,3\}$, $a \in {\cal A}$ we have that
$|\{x\in C:x_i=a \}|=\gamma q/2.$

2. \cite{KMP}  A code $C$ is a completely regular code with
$\rho=1$ and  eigenvalue $\lambda_3(3,q)$ in $H(3,q)$ if and only if  
 any maximal clique contains the same number of vertices in $C$.
Such codes with the parameter $\gamma$ exist if and only if $\gamma$ is  divisible by $3$.

\end{proposition}
  The Proposition \ref{p:OA}.2 is a folklore property which essentially states that a set is a completely regular code with the covering radius $1$ and the minimum eigenvalue if and only if it is evenly distributed by the Delsarte cliques in the Delsarte clique graphs. A detailed proof could be found in \cite[Theorem 2(ii), Example 2]{KMP}, where this class of codes is shown to coincide with multifold MDS codes in Hamming graphs.

\section{Constructions of completely regular codes in $3$-dimensional Hamming graphs}\label{S:Constr}

{\bf Construction A (adding nonessential positions to a CRC in $H(2,q)$). }

\begin{proposition}\label{Construction:A}
For every CRC in $H(2,q)$   with $\rho=1$,  eigenvalue $\lambda_2(2,q)$, and $\gamma$ there is a CRC of $H(3,q)$ with $\rho = 1$   eigenvalue  $\lambda_2(3,q)$, and $\gamma$. For every even $\gamma$   there is a CRC  in $H(3,q)$ with $\rho=1$,   eigenvalue $\lambda_2(3,q)$
 and the parameter $\gamma$. \end{proposition}
\begin{proof}

By Proposition \ref{P:2}.2 there are CRCs in $H(2,q)$ with any even $\gamma$, $\rho=1$ and eigenvalue $\lambda_2(2,q)=-2$. Then we add a nonessential position in order to obtain a desired code in $H(3,q)$.
The statement follows from Proposition \ref{P:3}.
\end{proof}

\begin{proposition}\label{p:odd q}
Let $q$ be odd.
A CRC in $H(3,q)$ with $\rho=1$, eigenvalue $\lambda_2(3,q)$, and $\gamma$
exists if and only if 
$\gamma$ is a positive even integer. 
\end{proposition}
\begin{proof} By Proposition \ref{P:1}, the size of code $C$ is 
$\frac{q^3\gamma}{\gamma+\beta}$ and
 $\lambda_2(3,q)=3q-3-(\gamma+\beta)=q-3$ is an  eigenvalue of the code $C$. So, we have
$\gamma+\beta=2q$, which given that $\frac{q^3\gamma}{\gamma+\beta}$ is integer and $q$ is odd implies that  $\gamma$ is necessarily even. On the other hand, for any  
positive even 
 $\gamma$ a desired code exists by Construction A.

\end{proof}

{\bf Construction B (alphabet lifting of perfect codes in $H(3,2)$).}
\begin{proposition}
    Let $q$ be even. Then there is a completely regular code  in $H(3,q)$ with $\rho=1$, eigenvalue $\lambda_2(3,q)$, and $\gamma=q/2$.
\end{proposition}
\begin{proof}
It is easy to find CRCs in $H(3,2)$ with eigenvalue $\lambda_2(3,2)$ with $\gamma=1$ or $\gamma = 2$, e.g., codes $C=\{(0,0,0),(1,1,1)\}$ and $D=\{(0,0,0),(1,0,0),(1,1,1),(0,1,1)\}$, respectively.  
Now we consider the alphabet ${\cal A}$ of size $q$, $q$ is even and the codes in $H(3,q)$ which are obtained from the above CRCs $C$ and $D$ in $H(3,2)$ as follows:
$\{x: x\, mod\, 2 \in C\}$
and $\{x: x\, mod\, 2 \in D\}$, where modulo for a tuple of length $3$ is applied symbol-wise.

Using~\cite[Theorem 4.7]{BKMTV}, for every even $q$, we see that the  codes above are completely regular codes in $H(3,q)$  with $\gamma=\frac{q}{2}$, $\rho = 1$, and eigenvalue $\lambda_2(3,q)$.

\end{proof}

{\bf Construction C (book construction).}

\begin{theorem}
For any even  $q$ there are completely regular codes  in $H(3,q)$  with $\rho=1$, eigenvalue $\lambda_{2}(3,q)$, and any $\gamma$, $\frac{q}{2}<\gamma<q$.
\end{theorem}

\begin{proof}

Let $T$ and $S$  be the following subsets of the alphabet set $\cal{ A}$: $T=\{0,\ldots,t-1\}$ and $S = \{ 0,\ldots,q/2-1 \}$, where $t > q/2$.  We also define  $\overline{T} = \mathcal{A} \setminus T$ and  $\overline{S} = \mathcal{A} \setminus S$.

 We choose a completely regular code    $D$  in $H(2,t)$ with $\rho=1$,  eigenvalue $-2$ and parameters $\gamma'=2t-q$ and $\beta'=q$. The existence of such a code follows from Proposition \ref{P:2}.2 because $q$ is even.   Note that every maximal clique of $H(2,t)$ contains $t - q/2$ vertices of $D$ and $q/2$ vertices  of $\overline{D}$.
 
Define the code $C$ in $H(3,q)$ as
\begin{gather*}
C = (T \times  \overline{T} \times S) \cup  (\overline{T} \times T \times  \overline{S}) \cup (D \times \mathcal{A});  \\
\overline{C} = \mathcal{A}^3 \setminus C,
\end{gather*}
where  $D$ is considered as a subset of $T \times T$. We illustrate the construction in Figure 1.
Let us prove that $C$ is a CRC with  $\gamma = t$ and $\beta=2q-t$.

Each vertex of $C$ from $T \times  \overline{T} \times S$ is adjacent to $q-t$ vertices  of $\overline{C}$ from $\overline{T} \times  \overline{T} \times S$, $q/2$ vertices of $\overline{C}$ from  $T \times  \overline{T} \times \overline{S}$  and $q/2$  vertices of $\overline{C}$ from $T \times T \times S$. So it is adjacent to $2q-t$ vertices of $\overline{C}$ in total.
By similar reasoning,  each vertex of $C$ from $\overline {T} \times  T \times \overline{S}$ is adjacent to $2q - t$ vertices of $\overline{C}$.
\begin{table}[H]
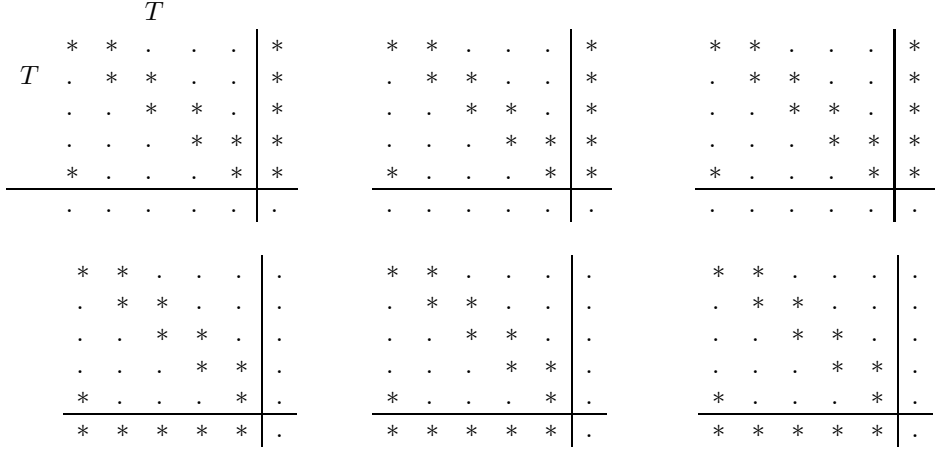
  \caption*{{\bf Figure 1.} {\it Construction C:  a CRC  with $\gamma=5$ in $H(3,6)$ represented via hyperfaces of direction $3$. Two vertices are adjacent in $H(3,6)$ if they are in the same row/column or coincide after overlaying of hyperfaces. 
 The first three hyperfaces correspond to the set $S$. The vertices of $C$ are denoted as $*$.
 The code closes the open case in  \cite[Table A.3, line 8]{BKMTV}.
 }
 
 }
$\begin{array}{llllll|l}
       &     &  &T &  \\ 

       & *    &*&  . &. &.&*\\ 
           T   &   .  &*&  * &. &.&*\\ 
                     & .    &.&  * &* &.&*\\ 
       & .    &.&  .&*&*&*\\ 
       & *    &.&  . &.&*&* \\ 
        \hline
       & .    &.&  .&.&.&.\\ 
        \end{array} 
\hspace{1cm}\begin{array}{llllll}
         &   &&&\\
           * &*&  . & .&.& \multicolumn{1}{| c }{*}\\ 
           . &*&  * & .&.& \multicolumn{1}{| c }{*}\\ 
           . &.&  * & *&.& \multicolumn{1}{| c }{*}\\ 
           . &.&  . & *&*& \multicolumn{1}{| c }{*}\\ 
           * &.&  . & .&*& \multicolumn{1}{| c }{*}\\ 
         
        \hline
          . &.&  .& .&  . &\multicolumn{1}{| c }{.}\\ 
        \end{array}$ 
$\hspace{1cm}\begin{array}{llllll}
         &   &&&\\
           * &*&  . & .&.& \multicolumn{1}{| c }{*}\\ 
           . &*&  * & .&.& \multicolumn{1}{| c }{*}\\ 
           . &.&  * & *&.& \multicolumn{1}{| c }{*}\\ 
           . &.&  . & *&*& \multicolumn{1}{| c }{*}\\ 
           * &.&  . & .&*& \multicolumn{1}{| c }{*}\\ 
         
        \hline
          . &.&  .& .&  . &\multicolumn{1}{| c }{.}\\ 
        \end{array}$ 

$ \hspace{0.75cm}
\begin{array}{llllll}
         &   &&&\\
           * &*&  . & .&.& \multicolumn{1}{| c }{.}\\ 
           . &*&  * & .&.& \multicolumn{1}{| c }{.}\\ 
           . &.&  * & *&.& \multicolumn{1}{| c }{.}\\ 
           . &.&  . & *&*& \multicolumn{1}{| c }{.}\\ 
           * &.&  . & .&*& \multicolumn{1}{| c }{.}\\ 
         
        \hline
          * &*&  *& *&  * &\multicolumn{1}{| c }{.}\\ 
        \end{array}$ 
\hspace{0.9cm}$\begin{array}{llllll}
         &   &&&\\
           * &*&  . & .&.& \multicolumn{1}{| c }{.}\\ 
           . &*&  * & .&.& \multicolumn{1}{| c }{.}\\ 
           . &.&  * & *&.& \multicolumn{1}{| c }{.}\\ 
           . &.&  . & *&*& \multicolumn{1}{| c }{.}\\ 
           * &.&  . & .&*& \multicolumn{1}{| c }{.}\\ 
         
        \hline
          * &*&  *& *&  * &\multicolumn{1}{| c }{.}\\ 
        \end{array}$  \hspace{1cm} 
$ \begin{array}{llllll}
         &   &&&\\
           * &*&  . & .&.& \multicolumn{1}{| c }{.}\\ 
           . &*&  * & .&.& \multicolumn{1}{| c }{.}\\ 
           . &.&  * & *&.& \multicolumn{1}{| c }{.}\\ 
           . &.&  . & *&*& \multicolumn{1}{| c }{.}\\ 
           * &.&  . & .&*& \multicolumn{1}{| c }{.}\\ 
         
        \hline
          * &*&  *& *&  * &\multicolumn{1}{| c }{.}\\ 
        \end{array}$ 
\hspace{1cm}
\end{table}

Consider a vertex of $C$ from $D \times \mathcal{A}$. It is adjacent to $q$ vertices of $\overline{C}$ from $\overline{D} \times \mathcal{A}$ and $q - t$ vertices of $\overline{C}$ from  $T \times  \overline{T} \times \overline{S}$ or $\overline{T} \times  T \times S$.  So it is also adjacent to $2q-t$ vertices of $\overline{C}$ in total.

Let us look at vertices from $\overline{C}$ now. Each vertex of $\overline{C}$ from $T \times  \overline{T} \times \overline{S}$ is adjacent $q/2$ vertices of $C$ from  $T \times  \overline{T} \times S$ and $t - q/2$ vertices   of $C$ from  $T \times  T \times \mathcal{A}$, so it is adjacent to $t$ vertices in total. By similar reasoning,  each vertex of $\overline{C}$ from $\overline {T} \times  T \times S$ is adjacent to $ t$ vertices of $C$.

Consider a vertex of $\overline{C}$ from $\overline{D} \times \mathcal{A}$. It is adjacent to $2t - q$ vertices of $C$ from $D \times \mathcal{A}$ and $q - t$ vertices of $C$ from  $T \times  \overline{T} \times S$ or $\overline{T} \times  T \times \overline{S}$.  So it is also adjacent to $t$ vertices of $C$ in total. 
Finally, each vertex of $\overline{C}$ from  $\overline{T} \times  \overline{T} \times \mathcal{A}$ is adjacent to $t$ vertices of $C$ from  $T \times  \overline{T} \times S$ or $\overline{T} \times  T \times \overline{S}$. Thus,  $C$ is a CRC  with intersection array $\{2q-t;t\}$ and has eigenvalue $\lambda_2(3,q)$.

\end{proof}

{\bf Construction D }

Consider the following condition for $q$ and $\gamma$.

{\bf Condition 1:} There are  integers $r$, $s$, $t$, $0< r, s, t<q$ and integers $a$, $b$, $c$, $0 < c \leq \min\{ q-s, q-t\}, 0 < b \leq \min\{ t, q-r\}, 0 < a \leq \min\{ r, s\}$ such that:   
$$\left\{  \begin{array}{l}
cr = a (q - t) \\
b(q - s) = c (q - r) \\
at = bs\\
\gamma=a+b+c.\\
\end{array} \right.$$

\begin{theorem}\label{T:orthg const}
If Condition 1 holds then there is a CRC in $H(3,q)$ with covering radius $1$,    eigenvalue $\lambda_2(3,q)$, and $\gamma$.
\end{theorem}

\begin{proof}
Let $R$, $S$, $T$ be subsets of the alphabet set ${\cal A}$, $|{\cal A}| = q$, $|R|=r$, $|S|=s$, $|T|=t$.  We identify the  vertex sets of the complete graphs $H(1,r)$, $H(1,s)$, $H(1,t)$, $H(1,q-r)$, $H(1,q-s)$, and $H(q-t,1)$ with the sets $R$, $S$, $T$, $\overline{R}$, $\overline{S}$,  and $\overline{T}$, respectively.

Further, we choose the sets $D^1\subseteq S\times T$, $D^2\subseteq R\times \overline{T}$, $D^3\subseteq \overline{R}\times \overline{S}$ such that \begin{itemize}
\item $D^1$ is a $(a,b)$-stochastic set in  $H(1,s) \Box H(1,t)$;
\item $D^2$ is a $(a,c)$-stochastic  set in  $H(1,{r}) \Box H(1,{q-t})$;
\item $D^3$ is a $(b,c)$-stochastic  set in  $H(1,{q-r}) \Box H(1,{q-s})$.
\end{itemize} Due to the definition of a stochastic set, we see that for each $i\in\{1,2,3\}$ the set $D^i$  either coincides with the corresponding graph or $D^i$ is a CRC in the graph. Proposition~\ref{P:2} and the conditions of the theorem imply that such codes exist and have eigenvalue $-2$. 
In particular, we have that
$$(q-s)tr = s(q-t)(q-r).$$

Consider the set $C = C^1 \cup C^2 \cup C^3$ in $H(3,q)$, where
\begin{gather*}
C^1 = \{(x_1,x_2,x_3): x_1 \in {\cal A}, x_2 \in S,  x_3 \in T,  (x_2,x_3)\in D^1\};\\
C^2 =\{(x_1,x_2,x_3): x_1 \in R, x_2 \in {\cal A}, x_3 \in \overline{T}, (x_1,x_3)\in D^2\}; \\
C^3 =\{(x_1,x_2,x_3): x_1 \in \overline{R}, x_2 \in \overline{S},  x_3 \in {\cal A}, (x_1,x_2)\in D^3\}.
\end{gather*}

Let us prove that $C$ is a CRC in $H(3,q)$ with $\beta = 2q - (a + b +c)$, $\gamma = a + b + c$, and therefore with eigenvalue $\lambda_2(3,q)$. For this aim we count in Table 1  the numbers of the vertices of $\overline{C}$ from a block $V$ (column) of the graph $H(3,q)$ adjacent to any fixed vertex of set $C$ from  a block $U$ (row).

$$
\begin{array}{|c||c|c|c|c|}
\hline
U / V & R \times \overline{S} \times T &  \overline{R} \times \overline{S} \times T &  R \times S \times T &  R \times \overline{S} \times \overline{T}   \\
\hline
\hline
R \times S \times T & q-s & ~ & (s - a) + (t -b) & ~  \\ 
\hline
\overline{R} \times S \times T & ~ & q-s - c & ~ & ~ \\ 
\hline
\hline
R \times \overline{S} \times \overline{T}  & t & ~ & ~ & (r-a) + (q - t - c)  \\ 
\hline
R \times S \times \overline{T}  & ~ & ~ & t - b & ~ \\ 
\hline
\hline
\overline{R} \times \overline{S} \times T  & r & (q - r - b) + (q-s-c) & ~ & ~ \\ 
\hline
\overline{R} \times \overline{S} \times \overline{T}  & ~ & ~ & ~ & r - a \\ 
\hline
\end{array}
$$

$$
\begin{array}{|c||c|c|c|c|}
\hline
U / V &   \overline{R} \times S \times T &  \overline{R} \times \overline{S} \times \overline{T} &  R \times S \times \overline{T} & \overline{R} \times S \times \overline{T} \\
\hline
\hline
R \times S \times T & ~ & ~ & q - t - c & ~ \\ 
\hline
\overline{R} \times S \times T & (s - a) + (t -b) & ~ & ~ & q - t \\ 
\hline
\hline
R \times \overline{S} \times \overline{T}  & ~ & q - r - b & ~ & ~  \\ 
\hline
R \times S \times \overline{T}  & ~ & ~ & (r-a) + (q - t - c) & q-r \\ 
\hline
\hline
\overline{R} \times \overline{S} \times T  & s -a & ~  & ~ & ~ \\ 
\hline
\overline{R} \times \overline{S} \times \overline{T}  & ~ & (q - r - b) + (q-s-c)  & ~ & s \\ 
\hline
\end{array}
$$
\begin{center}
Table 1. The numbers of the vertices from $V \cap \overline{C}$ adjacent to a given vertex of  $U \cap C$. The blank spaces are zeros. 
\end{center}

The entries are verified in a straightforward manner. 
For example, a vertex of $R\times S\times T$ is adjacent to $q-s$ vertices of $R\times \overline{S}\times T$ which are obtained by replacing the symbol in the second position of the vertex with any of $q-s$ symbols of $\overline{S}$. Whereas a vertex from $(R\times S\times T) \cap C$ is adjacent to exactly $s-a$ and $t-b$ vertices of $(R\times S\times T) \cap \overline{C}$ due to $C$ being $(a,b)$-stochastic in $S\times T$.

Since the sum of the entries in each row of Table 1 is $2q - (a + b +c)$, we see that  $\beta = 2q - (a + b +c)$. 

Let us now verify that $\gamma$ is $a + b + c$. For this purpose we count the numbers of vertices of $C$ from a block $V$ of the graph $H(3,q)$ adjacent to any fixed  vertex of the set $\overline{C}$ from  a block $U$.

$$
\begin{array}{|c||c|c|c|c|c|c|}
\hline
U / V &  \overline{R} \times \overline{S} \times T &  R \times S \times T &  R \times \overline{S} \times \overline{T}   &   \overline{R} \times S \times T &  \overline{R} \times \overline{S} \times \overline{T} &  R \times S \times \overline{T}  \\
\hline
\hline
R \times \overline{S} \times T & b & a & c  & ~ & ~ & ~  \\ 
\hline
\overline{R} \times \overline{S} \times T & c + b & ~ & ~ & a & ~ & ~ \\
\hline
R \times S \times T & ~ & b + a & ~ & ~ & ~ & c \\
\hline
R \times \overline{S} \times \overline{T} & ~ & ~ & c + a  & ~ & b & ~\\
\hline
\overline{R} \times S \times T & c & ~ & ~ & b+a & ~ & ~ \\
\hline
\overline{R} \times \overline{S} \times \overline{T} & ~ & ~ & a & ~ & c + b & ~ \\
\hline
R \times S \times \overline{T} & ~ & b & ~ & ~ & ~ & c + a  \\
\hline
\overline{R} \times S \times \overline{T} & ~ & ~ & ~ & b & c & a \\
\hline
\end{array}
$$
\begin{center}
Table 2. The numbers of the vertices from $V \cap C$ adjacent to a given vertex of  $U \cap \overline{C}$. The blank spaces are zeros. 
\end{center}

Since the sum of the entries in each row of Table 2 is $a + b +c$, we see that  $\gamma = a + b +c$, and so $C$ is a CRC with eigenvalue $\lambda_2(3,q)$.

\end{proof}

\begin{remark}
Splitting construction II from~\cite{BKMTV} applied to $H(3,q)$ is a partial case  of Construction D, where we take  $r = s = t = q/2$ and $a = b = c$ ($q$ is even). In particular, from Corollary~4.17(1) of~\cite{BKMTV} we see that Construction D gives completely regular codes in $H(3,q)$ with $\rho=1$ and  eigenvalue $\lambda_2(3,q)$ for all even $q$ and $\gamma = 3 \tau$ or $\gamma = q/2 + 3\tau - 3$, $\tau = 1, \ldots, q/2$.
\end{remark}

\begin{remark}
\label{r:2}    
Construction D produces  many  codes with different parameters. In below we provide several examples which we find exotic as the illustration of the construction.

If we set $q=8$, $r=s=t=4$, $a=b=c=2$ and  $r=2$,  $s=4$, $t=6$ , $a=c=2$, $b=3$ we obtain codes with $\gamma=6$ and $7$, both greater than $\frac{q}{2}=4$. For $q=32$, $s=28$, $t=16$, $r=28$, $a=7$, $b=4$, $c=4$, we get $\gamma=a+b+c=15$. In this case the set $S\times T$ coincides with $D^1$,  whereas $D^2$ and $D^3$ are proper subsets of $R\times \overline{T}$ and $\overline{R}\times \overline{S}$. In other words, the upper bound $min\{t,q-r\}$ in Condition 1 is attained for  $b$.

Construction D also gives codes over odd alphabet $q$ and even $\gamma$ nonisomorphic to those  from Construction A. The smallest value of $q$ for which such code exists is $45$.  By taking $r=9$, $s=15$, $t=30$ and $a=3$, $b=6$, $c=5$, we obtain a CRC with $\gamma=14$. 
\end{remark}

\section{CRCs in  $H(3,q)$ with clique property}\label{S:CRCsclique}

We will say that a completely regular code $C$ with $\rho=1$ in $H(3,q)$ has a {\it clique property}  if  $C$ consists of disjoint cliques. 
 If, moreover, there are cliques of all three codirections in $C$, we say that $C$  fulfills a {\it strong clique property}.

Now we work towards showing that Construction D describes all completely regular codes  in $H(3,q)$ with a strong clique property. In this section we use only combinatorial arguments accompanied with even distribution of CRC by hyperfaces (Proposition \ref{p:OA}.1).

\begin{lemma}\label{l:str_clique_ev}  If $C$ is a completely regular code with $\rho=1$ in $H(3,q)$ fulfilling a strong clique property, then $\lambda_2(3,q)$ is its eigenvalue.

\end{lemma}
\begin{proof}
The nontrivial eigenvalues of a completely regular code  with $\rho=1$ in $H(3,q)$ is $\lambda_i(3,q)$ for some $i=1, 2, 3$. Let us exclude cases $i=1$ and $3$.

All codes   in $H(n,q)$ with eigenvalue $\lambda_1(n,q)$ were characterized in \cite{Mey}. All such codes for $n=3$ up to permutation of three positions are  $B\times {\cal A}\times {\cal A}$ for a proper nonempty subset $B$ of ${\cal A}$. 
 We see that these CRCs fulfill the clique property, but not the
 strong clique property because these codes do not contain cliques of codirection $1$.  

Let $C$ be a CRC with $\rho=1$ in $H(3,q)$ with eigenvalue $\lambda_3(3,q)$. Then due to Proposition \ref{p:OA}.2,  any  maximal clique $H(3,q)$ cannot consist solely of vertices of $C$.

\end{proof}

Let $C^i$  be a collection of cliques  in $H(3,q)$ of codirection $i$, $i=1, 2, 3$, such that $C^1$, $C^2$, $C^3$ are pairwise disjoint, and let $C=C^1\cup C^2\cup C^3$.
Given $C^1$, let $S$ and $T$ be the minimal (by inclusion) sets in ${\cal A}$ such that 

\begin{equation}\label{Cleq:1}C^1\subseteq {\cal A}\times S\times T.\end{equation}

Similarly, we consider  symbol sets $R$, $T'$, $R'$ and $S'$ such that

\begin{equation}\label{Cleq:2}C^2 \subseteq R\times {\cal A}\times T'.\end{equation}

\begin{equation}\label{Cleq:3}C^3\subseteq R'\times S' \times {\cal A}.\end{equation}

We show that the introduced sets  are tightly related.
\begin{lemma}\label{l:1} Let $C$ be a CRC with $\rho=1$ in $H(3,q)$, where $C=C^1\cup C^2\cup C^3$  and  $S$, $S'$, $T$, $T'$, $R$, $R'$ are minimal sets such that (\ref{Cleq:1}), (\ref{Cleq:2}), (\ref{Cleq:3}) hold.  Then   $R'=\overline{R}$, $T'=\overline{T}$, $S'=\overline{S}$. 
\end{lemma}
\begin{proof}
Without loss of generality,  it is sufficient to prove that $R'=\overline{R}$ (for other sets the proof is by a permutation of coordinate positions). Suppose the opposite and let $v$  be an element of ${\cal A}\setminus (R\cup R')$.

Let $u$ be an element of $R$. We compare the numbers of the vertices of $C$ in hyperfaces $v\times {\cal A}\times {\cal A}$ and $u\times {\cal A}\times {\cal A}$. From (\ref{Cleq:2}) and (\ref{Cleq:3}) we have that  $C^2\cup C^3\subset (R\times {\cal A}\times T') \cup (R'\times S' \times {\cal A}$).
This, taking into account that $v$ is not in $R\cup R'$, implies  that the hyperface $v\times {\cal A}\times {\cal A}$ does not contain vertices of $C^2$ and $C^3$. On the other hand, since $u\in R$, $R$ and $T'$ are minimal sets for which $C^2\subset R\times {\cal A}\times T'$,   
the hyperface $u\times {\cal A}\times {\cal A}$ contains at least one vertex of $C^2$. 

Since $C^1$ consists of cliques of codirection $1$, taking into account (\ref{eq:obvious}), we see that both hyperfaces $u\times {\cal A}\times {\cal A}$ and $v\times {\cal A}\times {\cal A}$ of direction $1$ contain the same number of vertices from $C^1$. From above considerations  we conclude that these hyperfaces contain different numbers of vertices from $C$. This contradicts Proposition \ref{p:OA}.1, which holds because by Lemma \ref{l:str_clique_ev} the code $C$ has the second eigenvalue.

\end{proof}

Let sets $D^1\subseteq S\times T$, $D^2\subseteq R\times \overline{T}$, and $D^3\subseteq \overline{R}\times \overline{S}$ be obtained from sets $C^1$, $C^2$, and $C^3$ by deleting positions $1$, $2$, $3$, respectively. We also denote by $\Gamma^1$, $\Gamma^2$ and, $\Gamma^3$ respectively the subgraphs of $H(2,q)$ induced by the sets $S\times T$, $R\times \overline{T}$, and $\overline{R}\times \overline{S}$, respectively.
\begin{lemma}\label{l:2}
Under the conditions of Lemma \ref{l:1}  the following hold.

 1. The set $D^1$ ($D^2$, $D^3$ ) either coincides with $S\times T$ ($R\times \overline{T}$,  $\overline{R}\times \overline{S}$) or it is a CRC in  $\Gamma^1$ ($\Gamma^2$, $\Gamma^3$) with $\rho=1$ and eigenvalue $-2$.

2. There are nonzero integers $a$, $b$, $c$ such that $D^1$ is $(a,b)$-stochastic in the graph $\Gamma^1$, $D^2$ is $(a,c)$-stochastic in $\Gamma^2$, and
$D^3$ is $(b,c)$-stochastic in $\Gamma^3$.

\end{lemma}
\begin{proof} Firstly, we note that any hyperface  with direction $3$ and  the union of cliques $C^3$ of codirection $3$ meet in the same number of vertices regardless of the hyperface, see  (\ref{eq:obvious}). From (\ref{Cleq:3}) and Lemma \ref{l:1}, we see that $C^3\subseteq \overline{R}\times \overline{S}\times {\cal A}$, so \begin{equation}\label{eq_C3}|{\cal A}\times {\cal A}\times w \cap C^3|=|\overline{R}\times \overline{S}\times w \cap C^3|\mbox{ does not depend on } w\in {\cal A}.\end{equation}

1.  In view of  Proposition \ref{P:2} up to a permutation of coordinate positions, it is sufficient to show that $D^1$ is  stochastic, i.e that $|S\times v\cap D^1|$ is constant for all $v$ in $T$.  
Since $C^1$ is obtained from $D^1$   by adding nonessential position, we 
have 
$$|{\cal A}\times S\times v\cap C^1|=q|S\times v\cap D^1|$$ 
and we prove an equivalent statement:    $|{\cal A}\times S\times v\cap C^1|$
does not depend on $v$ in $T$.

 Consider the hyperface ${\cal A}\times {\cal A}\times v$ and count the number of vertices from $C^1$, $C^2$ and $C^3$ in it. 
By (\ref{Cleq:1}) and (\ref{Cleq:3}) we have that  $C^1\subseteq{\cal A}\times S\times T$, $C^3\subseteq \overline{R}\times \overline{S}\times {\cal A}$. Moveover, by (\ref{Cleq:2}) and Lemma \ref{l:1} we get $C^2\subseteq R\times {\cal A}\times \overline{T}$, so  the hyperface 
${\cal A}\times {\cal A}\times v$, $v\in T$, does not contain any vertices of $C^2$. We obtain that
  \begin{equation}\label{face_part_v}|{\cal A}\times {\cal A}\times v \cap C|=|{\cal A}\times S\times v \cap C^1|+|{\cal A}\times \overline{S}\times v \cap C^3|=q| S\times v \cap D^1|+|\overline{R}\times \overline{S}\times v \cap C^3|.\end{equation}
 
    By  Proposition \ref{p:OA}.1 the number of  vertices of $C$ in a hyperface is a constant regardless of hyperface, so    from (\ref{eq_C3}) and (\ref{face_part_v}) we see that   
$|{\cal A}\times S\times v\cap C^1|$ does not depend on $v$ in $T$. We obtain that $|S\times v \cap D^1|$ does not depend on $v$. Similarly to the above, changing notations and permuting two positions, we have that $| u\times T \cap D^1 |$ is a constant for $u\in S$. By Proposition \ref{P:2}.1, we see that $D^1$ is $S\times T$ or a CRC in the graph $\Gamma^1$ with eigenvalue $-2$ .

2. We are to show that the following hold:  
\begin{gather*}
\mbox{for any } v\in T,u \in \overline{T}\mbox{ } |S\times v \cap D^1|=|R\times  u \cap D^2|, \\
\mbox{for any } v\in S,u \in \overline{S}\mbox{ } | v \times T \cap D^1|=|\overline{R}\times u \cap D^3|, \\
\mbox{for any } v\in R,u \in \overline{R}\mbox{ } | v\times \overline{T} \cap D^2|=| u \times \overline{S} \cap D^3|.
\end{gather*}

It is sufficient  to prove that $D^1$ and $D^2$ are both $(a,*)$-stochastic, i.e., the equality $|S\times v\cap D^1|=|R\times u \cap D^2|$  holds  regardless of  
$v\in T$ and $u \in \overline{T}$. For this purpose we write an analogous equality to  (\ref{face_part_v}) but for the hyperface ${\cal A}\times {\cal A}\times u$, $u\in \overline{T}$:  \begin{equation}\label{face_part_u}|{\cal A}\times {\cal A}\times u \cap C|=|R \times {\cal A}\times u \cap C^2|+|{\cal A}\times \overline{S}\times u \cap C^3|=q| R\times u \cap D^2|+|\overline{R}\times \overline{S}\times u \cap C^3|.\end{equation}
 The  equalities  (\ref{eq_C3}), (\ref{face_part_v}), and (\ref{face_part_u}) imply that $|S\times v\cap D^1|=|R\times u \cap D^2|$, i.e., $D^1$ and $D^2$ are both $(a,*)$-stochastic. By permuting positions and renaming $D^1$, $D^2$, $D^3$, we obtain that $D^1$ and $D^3$ are $(*,b)$-, $(b,*)$-stochastic,  $D^2$ and $D^3$ are $(*,c)$-, $(*,c)$-stochastic. Therefore,
$D^1$, $D^2$, and $D^3$ are $(a,b)$-stochastic, $(a,c)$-stochastic, and
$(b,c)$-stochastic
respectively.

\end{proof}

From Lemma  \ref{l:2}.2 we obtain the following.

\begin{theorem}\label{t:char}
A CRC code    in $H(3,q)$, $q\geq 2$, with $\rho=1$ fulfills strong clique property if and only if it is  obtained by Construction D.
\end{theorem}

\begin{remark}
 It can be shown that the CRCs of $H(3,q)$ fulfilling the clique property but not strong clique property are exactly the codes with $\lambda_1(3,q)$ (characterized in \cite{Mey}) and the codes with $\lambda_2(3,q)$ obtained by permutation switching construction \cite{MV}. We  skip the proof of this fact for the sake of brevity.     
\end{remark}

\section{CRCs with the second eigenvalue and small odd $\gamma$}
\label{S:Deriv}
For a set $C$ of vertices of $H(3,q)$, we denote by $\chi_C$ its characteristic function.

Let $i\in \{1,2,3\}$ and $u$, $v$ be symbols of ${\cal A}$. Consider  tuples $x$ and $y$  from ${\cal A}^3$ that coincide in all positions except $i$th, $x_i=u$ and $y_i=v$. Denote by $x'$ the tuple  obtained from $x$ by deleting the $i$th position. Consider the  function $(\chi_C)_{i,u,v}$ defined as follows: 
\begin{equation}
\label{eq_star}(\chi_C)_{i,u,v}(x')= \chi_C(x)-\chi_C(y).\end{equation}

The function $(\chi_C)_{i,u,v}$ with domain ${\cal A}^2$ takes  values $\{0,1,-1\}$. If $C$ is a completely regular code in $H(3,q)$ with covering radius $1$ and the second eigenvalue, then  the function $(\chi_C)_{i,u,v}$ is an eigenvector of $H(2,q)$ with  eigenvalue $\lambda_2(2,q)$  \cite{MV} and, moreover,  these functions are characterized. We  note that a generalization  for this fact holds for all CRCs in $H(n,q)$  regardless of their eigenvalues,  see \cite{V1}  and \cite[Corollary 1]{MV}.

Let $X, Y$ be disjoint nonempty subsets of $
{\cal A}$,  $|X|=|Y|$, and $j$ be  $1$ or $2$. Consider the following function $f$, defined on ${\cal A}^2$, which we call the $(X,Y,j)$-{\it 
string}:

$$f(x)=
\begin{cases}
    1, & x_j \in X \\
    -1, & x_j \in Y \\
    0, & \hbox{otherwise}
\end{cases}.$$
In case when we are not interested in $X,Y,j$,
we will abbreviate this notation and refer to this function as a {\it string}.

Given  two nonempty proper  subsets $X,
Y$, $|X|=|Y|$ of  ${\cal A}$, define the following function $f$ on ${\cal A}^2$, which we call the
$(X,Y)$-{\it  cross}:
$$f(x)=
\begin{cases}
    1, & x_1 \in X, x_2 \notin Y\\
    -1, & x_1 \notin X, x_2 \in Y  \\
    0, & \hbox{otherwise}
\end{cases}.$$
In case when we are not interested in  $X$, $Y$, we will abbreviate this notation and refer to this function as a {\it cross}. Both functions were introduced earlier in \cite{MV}. 

\begin{lemma}\label{l:orig_deriv}\cite[Corollary 1, Lemma 4]{MV}
Let $C$ be a completely regular code in $H(3,q)$ with $\rho=1$ and eigenvalue $\lambda_2(3,q)$. Then for every $i\in \{1,2,3\}$ and $u$, $v\in {\cal A}$ the function $(\chi_C)_{i,u,v}$ is a cross, a string or
the all-zero function.

\end{lemma}

\begin{lemma}\label{L1}
 Let $C$ be a CRC in $H(3,q)$ with $\rho=1$, eigenvalue $\lambda_2(3,q)$, and  
 $\gamma<q/2$. Then any two maximal cliques of $H(3,q)$ consisting only of vertices of $C$ are disjoint.
\end{lemma}
\begin{proof}
Indeed, if two such cliques exist, then the intersection parameter $\alpha_0$ for the CRC $C$ is greater than $2q-2$. But due to Proposition \ref{P:1}, $\alpha_0$ is $\lambda_2(3,q)+\gamma$, which is less than $q-3+q/2=3q/2-3$, a contradiction.

\end{proof}

We now consider two  structural lemmas.
\begin{lemma}\label{l:string}
Let $C$ be a completely regular code in $H(3,q)$ with $\rho=1$ and  eigenvalue $\lambda_2(3,q)$, $(x_1,x_2,u)$ be in $C$. If there is $v\in {\cal A}$ such that $(\chi_C)_{3,u,v}$ is a string and $(\chi_C)_{3,u,v}(x_1,x_2)=1$, then $(x_1,x_2,u)$ is contained in a maximal clique of $H(3,q)$ solely of vertices of $C$. 
\end{lemma}
\begin{proof}  
Let $(\chi_C)_{3,u,v}$ be the $(X,Y,1)$-string for some $X$ and $Y$ (the result for $(X,Y,2)$ is obtained by the transposition of two positions). Firstly, given that  $(\chi_C)_{3,u,v}(x_1,x_2)=1$, by the definition of string, we obtain that $x_1\in X$, and furthermore $(\chi_C)_{3,u,v}(x_1,w)=1$ for any $w\in {\cal A}$. 
Secondly, by the definition (\ref{eq_star}) of the function $(\chi_C)_{3,u,v}$, we see that 
$$(\chi_C)_{3,u,v}(x_1,w)=\chi_C(x_1,w,u)-\chi_C(x_1,w,v)=1.$$

We see that $(x_1,x_2,u)$ is in the clique  $x_1 \times {\cal A} \times u$ solely of vertices of $C$. 

\end{proof}

\begin{lemma}\label{l:cross}
Let $C$ be a CRC in $H(3,q)$ with $\rho=1$, the eigenvalue $\lambda_2(3,q)$, and $\gamma<q/2$. If there are  $(x_1,x_2,u)\in C$ and $v\in {\cal A}$ such that $(\chi_C)_{3,u,v}$ is a cross and $(\chi_C)_{3,u,v}(x_1,x_2)=1$, then $(x_1,x_2,u)$ is contained in a maximal clique solely of vertices of $C$. 
\end{lemma}\begin{proof}
Let the function $(\chi_C)_{3,u,v}$ be the  $(X,Y)$-cross for some $X$ and $Y$, $|X|=|Y|$.   Since the cross  $(\chi_C)_{3,u,v}$ has value $1$ only on $X\times \overline{Y}$ we obtain that 

\begin{equation}\label{e:inc}X\times \overline{Y}\times u \subset C.\end{equation}
Moreover, since $(\chi_C)_{3,u,v}(x_1,x_2)=1$,  we see that $x_1\in X$, $x_2\in \overline{Y}$.

We now show that  $X\times Y\times u \subset C$ or $\overline{X}\times \overline{Y}\times u \subset C$ which, being combined with (\ref{e:inc}), will imply the desired. 
Suppose the opposite, i.e., there are $y\in (X\times Y\times  u) \cap \overline{C}$ and $z\in (\overline{X}\times \overline{Y}\times  u) \cap\overline{C}$. 

Due to (\ref{e:inc}), the vertex $y\in  X\times Y\times u $ is adjacent to at least $|{\overline Y}|$ vertices of $C$ belonging to $X\times \overline{Y}\times u$. The vertex $z\in \overline{X}\times \overline{Y}\times u$ is adjacent to at least $|X|$ vertices of $C$ from $X\times \overline{Y}\times u$. Since both $y$ and $z$  do not belong to $C$,  we conclude that $\gamma\geq max(|X|,q-|Y|)=max(|X|,q-|X|)\geq q/2$, which contradicts the condition $\gamma<q/2$ of the lemma.

\end{proof}

\begin{theorem}\label{T:smalldeg}
Let $C$ be a  completely regular code in $H(3,q)$, $q\geq 2$, with $\rho=1$, eigenvalue $\lambda_2(3,q)$ and odd $\gamma$, $\gamma<q/2$. Then $C$ fulfills strong clique property and therefore is obtained by Construction D.
\end{theorem}
\begin{proof}
 We show that $C$ fulfills  clique property. Let  $x= (x_1,x_2,u)$ be a vertex in $C$.
We now consider possible cases for the value of $(\chi_C)_{3,u,v}(x_1,x_2)$. Due to definition (\ref{eq_star}) and $x$ being in $C$, the  value $(\chi_C)_{3,u,v}(x_1,x_2)$ is either $0$ or $1$. 
Consider several cases. 

Case a. Suppose that for all $v\in {\cal A}$ 
we have that
$(\chi_C)_{3,u,v}(x_1,x_2)=0$. By definition of the function $(\cdot)_{3,u,v}$, see (\ref{eq_star}), we get that $\chi_C(x_1,x_2,u)-\chi_C(x_1,x_2,v)=0$, i.e., $x$ is in the clique $x_1\times x_2\times {\cal A}$ solely of vertices in $C$.

Case b. Assume that for some $v \in {\cal A}$ we have $(\chi_C)_{3,u,v}(x_1,x_2)=1$. We see that the function $(\chi_C)_{3,u,v}$ is not all-zero, so by Lemma \ref{l:orig_deriv} it is a string or a cross.
In both cases  from Lemma \ref{l:string} and Lemma \ref{l:cross} we see that $x$ is contained in a clique solely of the vertices from $C$. 

Thus, we proved that any $x\in C$ is contained in a clique of vertices from $C$.
By Lemma \ref{L1}, there is a unique such clique and we obtain that $C$ consists of disjoint cliques. 

 Suppose that  strong clique property does not hold for the code $C$. Up to a permutation of coordinate positions, let $C$ be either a union of disjoint maximum cliques having direction $1$ or a union of disjoint maximum cliques having directions $1$ and $2$. 
For any $v\in  {\cal A}$ consider the set $\{x:x \in C, x_3=v \}$. By our hypothesis, this set consists of $t$ cliques having  directions $1$ or $2$ and has size $tq$.   However, by Proposition \ref{p:OA},  the size of $\{x:x\in C, x_3=v \}$ is $\gamma q/2$. We conclude that $\gamma$ is even, which contradicts the conditions of the theorem.

A description of the code by Construction D follows from Theorem \ref{t:char}.

\end{proof}

\section{The parameters of CRCs with covering radius $1$ and the second eigenvalue in $H(n,q)$}\label{S:CRCsparam}

\begin{theorem}\label{C:H3q} For any $q\geq 2$,
a CRC with $\rho=1$,  eigenvalue $\lambda_2(3,q)$,  and $\gamma$ in $H(3,q)$
 exists if and only if one of the following holds:

1.  $\gamma$ is even,  $2\leq \gamma\leq q$,

2. $q$ is even, $\frac{q}{2}\leq \gamma\leq q$, 

3. $q$ is even, $\gamma$ is odd, $\gamma<\frac{q}{2}$, and Condition 1    holds for $q$ and $\gamma$.

\end{theorem}
\begin{proof}
For odd $q$ the result follows from  Proposition \ref{p:odd q}.

Let $q$ be even. The existence of all CRCs with desired eigenvalue and positive even $\gamma$ or  $\frac{q}{2}\leq \gamma\leq q$ follows from Constructions A, B, C. By Theorem \ref{T:smalldeg} all CRCs with odd $\gamma$ less then $\frac{q}{2}$ are obtained by Construction D. In terms of parameters, it is equivalent to Condition 1.

\end{proof}

We note that there are instances of nonisomorphic completely regular codes with the same parameters in $H(3,q)$, see Remark \ref{r:2}. The characterization of completely regular codes up to isomorphism is still open in general and we have only partial results in Theorem \ref{T:smalldeg}. Given the characterization of completely regular codes with $\rho=1$ with the first \cite{Mey} and the third eigenvalue of $H(3,q)$ (see Proposition \ref{p:OA}.2), we get the following.  

\begin{theorem}
For any $q\geq 2$, a completely regular code $C$ in $H(3,q)$ with covering radius $1$ and  $\gamma$ exists if and only if one of the following holds.

1.\cite{Mey} The code $C$ has the eigenvalue $\lambda_1(3,q)$ and any $\gamma$, $0<\gamma\leq q/2$.

2. The code $C$ has the eigenvalue $\lambda_2(3,q)$, $q$, and $\gamma$ are described in Theorem \ref{C:H3q}.

3. The code $C$ has the  eigenvalue $\lambda_3(3,q)$ and any nonzero $\gamma$ divisible by $3$.

\end{theorem}

We visualize the results of this work and \cite{MV} in Table 3 and deduce the following. 
\begin{table}
\caption*{Table 3. Constructions of CRCs with $\rho=1$ in $H(n,q)$ with the second eigenvalue}
\begin{center}
\begin{tabular}{|c|c|c|c|c|c|c|c}
\hline
\backslashbox{Parameters}{Construction} &  A&  B&
  C&
   D&
   \makecell{ Permutation \\ switching }  & \makecell{ Alphabet \\ lifting }   \\
  
   \hline   

 $\gamma$ even, $2\leq \gamma\leq q$, any $q$ & $\star$ &  &+&+, Remark \ref{r:2}&$\oplus$&$\oplus$\\\hline
$\frac{q}{2}<\gamma<q$, $q$ even  &  &  &$\star$&+, Remark \ref{r:2}&&\\
\hline
$\gamma=\frac{q}{2}$, $q$ even  &  & $\star$ & &&&\\

\hline
$\gamma<\frac{q}{2}$, $\gamma$,  $q$ even  &  &  &&$\ostar$& &\\
\hline
\makecell{ The number of \\ essential positions } & 2 & 3 & 3 & 3 & $3\leq\cdot \leq \frac{q}{2} +1$&4\\
\hline

\end{tabular}
\label{tab1}
\end{center}

$+$ exists for some of  parameters\\
$\star$ exists for all parameters\\
$\ostar$ characterizes all CRCs  with these parameters \\
$\oplus$  characterizes all CRCs with at least $4$ essential positions.

\end{table}

\begin{theorem}
\label{T:Hnq}
For any  $n$ and $q\geq 2$ a completely regular code with covering radius $1$, eigenvalue $\lambda_2(n,q)$, and $\gamma$ exists in $H(n,q)$   if and only if one of the following holds:

1. $n\geq 2$ and $\gamma$ is even.

2. $n\geq 3$, $q$   is even, $\frac{q}{2}\leq \gamma\leq q$.

3. $n\geq 3$, $q$ is even, $\gamma$ is odd, $\gamma\leq q/2$, and Condition 1    holds for $q$ and $\gamma$.

\end{theorem}
\begin{proof} Given a completely regular code $C$ in  $H(n,q)$ with $\rho=1$ and eigenvalue $\lambda_2(n,q)$,  consider the  reduced code $C'$ in $H(n'q)$
which is obtained from the 
original code by removing nonessential positions.
By Proposition \ref{P:3}, $C'$ 
 has  eigenvalue $\lambda_2(n',q)$, the same covering radius $1$  and, moreover, the same intersection array as $C$.
By Theorem \ref{T:reduce}, the reduced 
 codes with more than $3$ essential positions have even $\gamma$ and their  parameters are not new, because they are covered by those of Construction A.  
The reduced codes with $n'=2$ (arising from the grid graph $H(2,q)$) all have even $\gamma$ and are described in Construction A (the first statement of the theorem). For $n'=3$ we have the result due to Theorem \ref{C:H3q}.

\end{proof}

\end{document}